\documentclass[12pt]{amsart}
\usepackage{amsmath,amssymb,amsbsy,amsfonts,latexsym,amsopn,amstext,
                                               amsxtra,euscript,amscd,bm}
                    
\usepackage{url}
\usepackage[colorlinks,linkcolor=blue,anchorcolor=blue,citecolor=blue]{hyperref}
\usepackage{color}

\begin{document}
\newtheorem{problem}{Problem}
\newtheorem{theorem}{Theorem}
\newtheorem{lemma}[theorem]{Lemma}
\newtheorem{claim}[theorem]{Claim}
\newtheorem{cor}[theorem]{Corollary}
\newtheorem{prop}[theorem]{Proposition}
\newtheorem{definition}{Definition}
\newtheorem{question}[theorem]{Question}
\newtheorem{application}[theorem]{Application}
\newtheorem{remark}[theorem]{Remark}

\def\cA{{\mathcal A}}
\def\cB{{\mathcal B}}
\def\cC{{\mathcal C}}
\def\cD{{\mathcal D}}
\def\cE{{\mathcal E}}
\def\cF{{\mathcal F}}
\def\cG{{\mathcal G}}
\def\cH{{\mathcal H}}
\def\cI{{\mathcal I}}
\def\cJ{{\mathcal J}}
\def\cK{{\mathcal K}}
\def\cL{{\mathcal L}}
\def\cM{{\mathcal M}}
\def\cN{{\mathcal N}}
\def\cO{{\mathcal O}}
\def\cP{{\mathcal P}}
\def\cQ{{\mathcal Q}}
\def\cR{{\mathcal R}}
\def\cS{{\mathcal S}}
\def\cT{{\mathcal T}}
\def\cU{{\mathcal U}}
\def\cV{{\mathcal V}}
\def\cW{{\mathcal W}}
\def\cX{{\mathcal X}}
\def\cY{{\mathcal Y}}
\def\cZ{{\mathcal Z}}

\def\A{{\mathbb A}}
\def\B{{\mathbb B}}
\def\C{{\mathbb C}}
\def\D{{\mathbb D}}
\def\E{{\mathbb E}}
\def\F{{\mathbb F}}
\def\G{{\mathbb G}}
\def\I{{\mathbb I}}
\def\J{{\mathbb J}}
\def\K{{\mathbb K}}
\def\L{{\mathbb L}}
\def\M{{\mathbb M}}
\def\N{{\mathbb N}}
\def\O{{\mathbb O}}
\def\P{{\mathbb P}}
\def\Q{{\mathbb Q}}
\def\R{{\mathbb R}}
\def\S{{\mathbb S}}
\def\T{{\mathbb T}}
\def\U{{\mathbb U}}
\def\V{{\mathbb V}}
\def\W{{\mathbb W}}
\def\X{{\mathbb X}}
\def\Y{{\mathbb Y}}
\def\Z{{\mathbb Z}}

\def\ep{{\mathbf{e}}_p}
\def\em{{\mathbf{e}}_m}
\def\eq{{\mathbf{e}}_q}

\def\scr{\scriptstyle}
\def\\{\cr}
\def\({\left(}
\def\){\right)}
\def\[{\left[}
\def\]{\right]}
\def\<{\langle}
\def\>{\rangle}
\def\fl#1{\left\lfloor#1\right\rfloor}
\def\rf#1{\left\lceil#1\right\rceil}
\def\le{\leqslant}
\def\ge{\geqslant}
\def\eps{\varepsilon}
\def\mand{\qquad\mbox{and}\qquad}

\def\sssum{\mathop{\sum\ \sum\ \sum}}
\def\ssum{\mathop{\sum\, \sum}}
\def\ssumw{\mathop{\sum\qquad \sum}}

\def\vec#1{\mathbf{#1}}
\def\inv#1{\overline{#1}}
\def\num#1{\mathrm{num}(#1)}
\def\dist{\mathrm{dist}}

\def\fA{{\mathfrak A}}
\def\fB{{\mathfrak B}}
\def\fC{{\mathfrak C}}
\def\fU{{\mathfrak U}}
\def\fV{{\mathfrak V}}

\newcommand{\bflambda}{{\boldsymbol{\lambda}}}
\newcommand{\bfxi}{{\boldsymbol{\xi}}}
\newcommand{\bfrho}{{\boldsymbol{\rho}}}
\newcommand{\bfnu}{{\boldsymbol{\nu}}}

\def\GL{\mathrm{GL}}
\def\SL{\mathrm{SL}}

\def\Hba{\overline{\cH}_{a,m}}
\def\Hta{\widetilde{\cH}_{a,m}}
\def\Hb1{\overline{\cH}_{m}}
\def\Ht1{\widetilde{\cH}_{m}}
\def\tz{\widetilde z}
\def\tX{\widetilde X}
\def\tx{\widetilde x}
\def\tu{\widetilde u}
\def\tv{\widetilde v}
\def\tw{\widetilde w}

\def\flp#1{{\left\langle#1\right\rangle}_p}
\def\flm#1{{\left\langle#1\right\rangle}_m}
\def\dmod#1#2{\left\|#1\right\|_{#2}}
\def\dmodq#1{\left\|#1\right\|_q}

\def\Zm{\Z/m\Z}

\def\Err{{\mathbf{E}}}

\newcommand{\abs}[1]{ \left\lvert#1\right\rvert} 
\newcommand{\norm}[1]{\left\lVert#1\right\rVert} 

\newcommand{\comm}[1]{\marginpar{%
\vskip-\baselineskip 
\raggedright\footnotesize
\itshape\hrule\smallskip#1\par\smallskip\hrule}}

\def\xxx{\vskip5pt\hrule\vskip5pt}


\title{ A new sum-product estimate in prime fields}

\author{Changhao Chen, Bryce Kerr, and  Ali Mohammadi }

\address{Department of Pure Mathematics, University of New South Wales, Sydney, NSW 2052, Australia} 
\email{changhao.chen@unsw.edu.au}  \email{bryce.kerr@unsw.edu.au}

\address{
School of Mathematics and Statistics, University of Sydney, NSW 2006, Australia}
\email{alim@maths.usyd.edu.au}




\pagenumbering{arabic}

\begin{abstract} 
In this paper we obtain a new sum-product estimate in prime fields. In particular, we show that if $A\subseteq \F_p$ satisfies $|A|\le p^{64/117}$ then 
$$
\max\{|A\pm A|, |AA|\} \gtrsim |A|^{39/32}.
$$
Our argument builds on and improves some recent results of Shakan and Shkredov which use the eigenvalue method to reduce to estimating a fourth moment energy and the additive energy $E^+(P)$  of some subset $P\subseteq A+A$. Our main novelty comes from reducing the estimation of $E^+(P)$ to a point-plane incidence bound of Rudnev rather than a point line incidence bound of Stevens and de Zeeuw as done by Shakan and Shkredov.
\end{abstract}

\maketitle

\section{Introduction}
Let $p$ denote a prime number and $\F_p$ the finite field of order $p$. Given a subset $A\subseteq \F_p$ we define the sum set and product set of $A$ respectively by
$$A+A=\{a+b :  a,b\in A\} \quad \text{and}\quad AA=\{ab : a,b\in A\}.$$
The sum-product theorem in $\F_p$ due to Bourgain, Katz and Tao~\cite{BKT} states that for all $0<\varepsilon<1$ there exists some $\delta>0$ such that if $p^{\varepsilon}<|A|<p^{1-\varepsilon}$ then
\begin{align}
\label{eq:sumproductfp}
\max\{|AA|,|A+A|\}\ge |A|^{1+\delta},
\end{align}
and  Glibichuk and Konyagin \cite{GliKon} have shown that the condition $p^{\varepsilon}<|A|$ may be dropped.

The sum-product problem was first considered by Erd\H{o}s and Szemer\'{e}di~\cite{ES} over the integers whose work led to the conjecture that for any $\varepsilon>0$ and finite subset $A\subseteq \R$ we have 
\begin{align*}
\max\{|AA|,|A+A|\}\gg |A|^{2-\varepsilon},
\end{align*}
with implied constant depending only on $\varepsilon$. The sharpest sum-product result over $\R$ is due to Shakan~\cite{Sha}. 

By a construction due to Chang~\cite{Cha}, for any $N\leq p$ there exists a subset $A \subseteq \F_p$ with $|A| = N$ such that
\begin{equation}
\label{eqnSPUB}
\max\{|A+A|, |AA|\} \ll p^{1/2}N^{1/2},
\end{equation}
and hence the Erd\H{o}s and Szemer\'{e}di conjecture cannot be true in full generality in $\F_p$. We expect the  conjecture to be true in $\F_p$ if we restrict to sets of sufficiently small cardinality  and an active field of research is to determine the largest possible $\delta$ such that~\eqref{eq:sumproductfp} holds. The first explicit sum product result in $\F_p$ is due to Garaev~\cite{Gar} which has had a number of improvements  since, see~\cite{BG,KaSh,Li,Rud1}. A major breakthrough came from the work of Roche-Newton, Rudnev and Shkredov~\cite{RNRuShk} which is based on Rudnev's point plane incidence bound~\cite{Rud} and states that if $|A|\le p^{5/8}$ then
\begin{align}
\label{eq:sp1}
\max\{|A+A|,|AA|\}\gg |A|^{6/5}.
\end{align}
We note that the idea of applying geometric incidence estimates to sum-product type problems is due to Elekes~\cite{El}. Stevens and de Zeeuw have provided a different proof of the estimate~\eqref{eq:sp1} using their point line incidence bound. Recently, Shakan and Shkredov~\cite[Theorem~1.3]{ShaShk} have broken the $6/5$ barrier for the sum-product problem over $\F_p$ and have shown that if $|A|\le p^{3/5}$ then
\begin{align*}
\max\{|A\pm A|,|AA|\}\gtrsim |A|^{6/5+4/305}.
\end{align*}
We note that their condition for $|A|<p^{3/5}$ can be extended to $|A|<p^{2/3}$, see Remark \ref{rem:remark} for more details. See also~\cite{MPRRS} for variations on the sum-product theorem including sharper results for the few sums many products problem, see~\cite{MRSS} for the few products many sums problem and~\cite{MRS} for various other results related to expanders in prime fields.
\newline 

In this paper we obtain a new sum-product estimate over $\F_p$ which improves on the result of Shakan and Shkredov stated above. Our proof builds on techniques from~\cite{ShaShk} which use the eigenvalue method, see~\cite{Sh}, to reduce to estimating a fourth moment energy $E^+_4(A,B)$ which counts the number of solutions to the equations
$$a_1-b_1=a_2-b_2=a_3-b_3=a_4-b_4, \quad a_i\in A, \ \ b_i\in B,$$
and the additive energy $E^+(P)$ of some subset $P\subseteq A+A$. Shakan and Shkredov reduce both $E_4(A,B)$ and $E^+(P)$ to the point line incidence bound of Stevens and de Zeeuw and our improvement comes from estimating $E^+(P)$ via Rudnev's point plane incidence bound.
\subsection*{Asymptotic notation}
For positive real numbers $X$ and $Y$, we use $X\ll Y$ and $Y\gg X$ to imply existence of an absolute constant $C>0$ such that $X \leq CY$. We also use $X\lesssim Y$ and $Y\gtrsim X$ to mean that there exists an absolute constant $C>0$ such that $X\ll (\log{X})^CY.$

\section{Main results}

Our first result provides an improvement on the sum product estimate of Shakan and Shkredov~\cite[Theorem~1.3]{ShaShk}.
\begin{theorem}
\label{thm:one}
Suppose $A\subset \F_p $ satisfies  
$$|A|\le p^{64/117}.$$ Then we have  
$$
\max\{|A\pm A|, |AA|\} \gtrsim |A|^{39/32}.
$$
\end{theorem}
In the case of difference set we obtain an estimate of the same strength with weaker conditions on the cardinality of $A$.
\begin{theorem}
\label{thm:SP}
Suppose $A\subset \F_p $ satisfies $|A|\ll p^{32/55}$. Then we have 
$$
\max\{|A-A|, |AA|\} \gtrsim |A|^{39/32}.
$$
\end{theorem}

We can obtain sharper estimates in the case of iterated sumsets. The case $k=3$ below agrees with an estimate of Roche-Newton, Rudnev and Shkredov~\cite[Corollary~12]{RNRuShk}.
\begin{theorem}
\label{thm:two}
Let $k\ge 3$ be an integer and suppose $A\subseteq \F_p$ satisfies 
$$
|A|\le p^{(4-3\times 2^{-k})/(7-16\times2^{-k})}.
$$ 
Then we have 
\begin{align*}
\max\{|kA|,|AA|\} \gtrsim |A|^{(5-2^{3-k})/(4-3\times 2^{1-k})}.
\end{align*}
\end{theorem}

\section{Preliminaries}
Given subsets $A,B\subseteq \F_p$, let
\begin{align*}
r_{A\pm B}(x)=\left|\{ (a,b)\in A\times B \ : \ a\pm b=x \}\right|,
\end{align*}
and for an integer $k$ define
\begin{align*}
E_k^+(A,B)=\sum_{x\in A-B}r_{A-B}(x)^k.
\end{align*}
We sometimes  write   $\sum_{x}$  to represent $\sum_{x\in \F_p}$ for convenience when the context is clear. For $A\subset \F_p$, we let $A(x)$ denote the characteristic function of $A$. We can write $r_{A+B}(x)$ as the convolution of functions $A$ and $B$, that is 
$$r_{A+B}(x)=(A*B)(x).$$

We write simply $E_k^+(A)$ instead of $E_k^+(A,A)$ and use $E^{+}(A, B)$ to denote $E_{2}^{+}(A, B)$, which we refer to as the additive energy between $A$ and $B$. Note that $E_k^+(A,B)$ counts the number of solutions to the equations
$$
a_1-b_1=\dots=a_k-b_k, \quad a_1,\dots,a_k\in A, \quad b_1,\dots,b_k\in B.
$$

The following is due to Shkredov~\cite[Proposition~31]{Sh}, see also~\cite[Lemma~6.1]{ShaShk}.
\begin{lemma}
\label{lem:S}
For any subset $A\subset \mathbb{F}_{p}$ and any $P\subset A-A$ we have  
$$
\left(\sum_{x\in P} r_{A-A}(x) \right)^{8}\leq |A|^{8}E_{4}^{+}(A)E^{+}(P).
$$
Similarly, for any $ P \subset A+A$ the following holds
$$
\left(\sum_{x\in P} r_{A+A}(x) \right)^{8}\leq |A|^{8}E_{4}^{+}(A)E^{+}(P).
$$
\end{lemma}

We will also require a third moment estimate of Konyagin and Rudnev \cite[Corollary 10]{KR}.
\begin{lemma}
\label{lem:KR}
For any subset $A\subset \F_{p}$ we have 
$$
\frac{|A|^{8}}{|A-A|}\ll E^{+}_{3}(A)E^{+}(A, A-A).
$$
\end{lemma}

Next, we recall a variation of the Pl\"{u}nnecke-Ruzsa inequality, which can be found in \cite{KaSh}. 
\begin{lemma}
\label{lem:PlRu}
Let $X, B_1, \dots, B_k \subseteq \F_p$. Then for any $0 < \epsilon < 1$ there exists a subset $X^{'} \subseteq X$, with $|X^{'}| \geq (1-\epsilon)|X|$ such that $$|X^{'} +B_1 + \cdots + B_k| \ll_{\epsilon} \frac{|X + B_1| \cdots |X + B_k|}{|X|^{k-1}}.$$
\end{lemma}

The following  point-line incidence bound is due to Stevens and de Zeeuw~\cite{StevZeeuw}, see also~\cite[Lemma 12]{Shkredov}. 

\begin{lemma}\label{lem:SZ}
Let $P=X\times Y$ be a subset of $\F_{p}^{2}$ and $L$ be a collection of lines in $\F_{p}^{2}$. 
Then 
$$I(P, L)\ll |X|^{3/4}|Y|^{1/2}|L|^{3/4}+|L|+|P|+\frac{|L||P|}{p}.$$ 
\end{lemma}

\begin{remark}
Using Lemma~\ref{lem:SZ} and a technique due to Elekes~\cite{El}, as outlined in \cite[Corollary~9]{StevZeeuw}, one obtains the estimate 
$$\max\{|A+A|, |AA|\}\gg |A|^{6/5}$$ 
for any set $A\subset \F_p$ under the condition $|A| \ll p^{5/7}$. It is worth noting that this improves on the condition $|A| \leq p^{5/8}$, which was obtained in \cite{RNRuShk} and \cite{StevZeeuw}. Furthermore, by \eqref{eqnSPUB}, it is easy to see that this condition is optimal up to some constant.
\end{remark}

The following is due to Shakan and Shkredov~\cite[Proposition~3.1]{ShaShk} and is based on Lemma~\ref{lem:SZ}. We note that their condition on the cardinality  $|A|<p^{3/5}$ can be extend to $|A|<p^{2/3}$ and we provide details of this extension.
\begin{lemma}
\label{lem:SS} 
Let $A \subset \F_{p}$ satisfy $|A|<p^{\frac{2}{3}}$. Then for any subset $B\subset \F_{p}$ we have 
$$
E_{4}^{+}(A, B)\lesssim |B|^{3}|AA|^{2}|A|^{-1}.
$$
\end{lemma}
\begin{proof}
 Taking a dyadic decomposition of $r_{A-B}(x)$, there exists a real number $\tau$ such that defining
$$D_{\tau}=\{x\in A-B: \tau\leq r_{A-B}(x)<2\tau\},$$ 
we have 
\begin{align}
\label{eq:E4dyadic}
E_4^+(A,B)=\sum_{x}r_{A-B}(x)^{4}\lesssim \tau^{4}|D_{\tau}|,
\end{align}
and 
\begin{equation}\label{eq:dyadic}
\tau |D_{\tau}|\ll |A||B|, \quad  \tau^{2}|D_{\tau}|\ll E^+(A, B).
\end{equation}

Consider the set of points $P=D_{\tau}\times AA$ and the set of lines
$L=\{\ell_{a,b}: a \in A, b \in B\}$
where 
$$\ell_{a, b}=\{(x, y)\in \F_{p}^{2}: y=(x+b)a\}.$$  For any $a \in A$ and $b\in B$ we have 
$$|\ell_{a, b} \cap P |\geq \sum_{a_{1}\in A}{\bf 1}_{D_{\tau}}(a_{1}-b).$$
Thus we obtain
\begin{align*}
I(P, L)&=\sum_{a\in A, b\in B}| \ell_{a, b} \cap P|\\
&\geq \sum_{a\in A}\sum_{ a_{1}\in A, b\in B} {\bf 1}_{D_{\tau}}(a_{1}-b)\\
&=\sum_{a\in A}\sum_{x\in D_{\tau}} r_{A-B}(x)\\
&\gg  |A||D_{\tau}| \tau.
\end{align*}
Combining with  Lemma \ref{lem:SZ} we conclude that
\begin{equation}\label{eq:incidence}
\begin{aligned}
|A||D_{\tau}|\tau&\ll |D_{\tau}|^{3/4}|AA|^{1/2}(|A||B|)^{3/4}
\\&+|D_{\tau}||AA|+|A||B|+\frac{|D_{\tau}||AA||A||B|}{p}.
\end{aligned}
\end{equation}
We next proceed on a case by case basis depending on which term in \eqref{eq:incidence} dominates.
Suppose the first dominates, so that 
$$|A||D_{\tau}|\tau\ll |D_{\tau}|^{3/4}|AA|^{1/2}(|A||B|)^{3/4},$$
which gives the desired result after combining with~\eqref{eq:E4dyadic}.

Suppose that the second term in~\eqref{eq:incidence} dominates. This implies that 
$$|A||D_{\tau}|\tau \ll |D_{\tau}||AA|,$$
and hence $$\tau \ll |AA|/|A|.$$ Combining with \eqref{eq:dyadic} and using the trivial bound
$$E^+(A,B)\le |A||B|^2,$$
 we get
$$\tau^{4}|D_{\tau}| =\tau^{2} E^+(A,B)\ll |B|^{2}|AA|^{2}|A|^{-1}.$$

If the third term in~\eqref{eq:incidence} dominates, then we have 
$$\tau |D_{\tau}|\ll |B|,$$
so that using the trivial bound  $\tau\leq \min\{ |A|, |B|\}$, we obtain
$$\tau^{4}|D_{\tau}| = \tau^{3} \tau |D_{\tau}|\ll  |B|^{3}|A|\ll|B |^{3}|AA|^{2}|A|^{-1}.$$

Finally consider when  the last term in~\eqref{eq:incidence} dominates, so that
\begin{equation}\label{eq:tau}
p\tau\ll |B||AA|.
\end{equation}
If $$\tau \leq |AA||B||A|^{-3/2},$$ then 
$$
|D_{\tau}| \tau^{4}=|D_{\tau}|\tau^{2} \tau^{2}\ll |A|^{2}|B| |AA|^{2}|B|^{2}|A|^{-3},
$$
which gives the desired result. Otherwise, suppose $$\tau > |AA||B||A|^{-3/2},$$ which combined with \eqref{eq:tau} implies that 
$$
p|AA||B||A|^{-3/2}\ll |B||AA|,
$$ 
and contradicts our assumption $|A|<p^{2/3}.$
\end{proof}

\begin{remark}
\label{rem:remark}
Combining Lemma~\ref{lem:SS} with \cite[Theorem 2.5]{ShaShk} leads to the same sum-product estimate as~\cite[Theorem~1.3]{ShaShk} with the weaker condition $|A|<p^{2/3}$.  
\end{remark}

Using H\"{o}lder's inequality and Lemma \ref{lem:SS} we obtain the following third moment estimate which will be used in the proof of Theorem~\ref{thm:SP}.

\begin{lemma}
\label{lem:lower} 
For any subset $A\subset \F_p$ satisfying $|A|<p^{\frac{2}{3}}$ we have  
$$
E^{+}_{3}(A)\lesssim |AA|^{\frac{4}{3}}|A|^2.
$$
\end{lemma}
\begin{proof}

Writing 
$$
E^{+}_{3}(A)=\sum_{x}r_{A-A}(x)^{\frac{8}{3}+\frac{1}{3}},
$$
and applying  H\"older's inequality and Lemma \ref{lem:SS} gives
\begin{align*}
E^{+}_{3}(A)&\leq E^{+}_{4}(A)^{\frac{2}{3}}(|A||A|)^{\frac{1}{3}}\\
& \lesssim \left( |AA|^{2}|A|^{2}\right)^{\frac{2}{3}}|A|^{\frac{2}{3}},
\end{align*}
which is the desired result.
\end{proof}

The following is due to Roche-Newton, Rudnev and Shkredov~\cite[Theorem~6]{RNRuShk} and is based on Rudnev's point plane incidence bound~\cite{Rud}.
\begin{lemma}
\label{lem:ERRS}
Let $X, Y, Z \subset \F_p$ and let $M = \max\{|X|, |YZ|\}.$ Suppose that $|X||Y||YZ| \ll p^{2}$. Then we have
$$
E^{+}(X, Z) \ll (|X||YZ|)^{3/2}|Y|^{-1/2} + M|X||YZ||Y|^{-1}.
$$
\end{lemma}

\begin{cor}
\label{cor:ERRS}
Let $A \subset \F_p$. If $|A+A||AA||A|\ll p^{2}$ then 
\begin{align*}
\label{eqn:ERRS}
E^{+}(A, A+A)\ll (|A+A||AA|)^{3/2}|A|^{-1/2}.
\end{align*} 
Similarly, if $|A-A||AA||A|\ll p^{2}$, then
\begin{align*}
E^{+}(A,A-A) \ll&(|A-A||AA|)^{3/2}|A|^{-1/2}.
\end{align*}
\end{cor}
\begin{proof}
We consider only $A+A$, a similar argument applies to $A-A$. Applying Lemma \ref{lem:ERRS} with
$$X=A+A, \quad  Y=Z=A,$$
gives
\begin{align*}
E^{+}(A,A+A) \ll &(|A+A||AA|)^{3/2}|A|^{-1/2} +  |A+A|^{2}|AA||A|^{-1} \\  &+ |A+A||AA|^{2}|A|^{-1}.
\end{align*}
Observe that  for any subset $A\subset \F_p$ we have
$$(|A+A||AA|)^{3/2}|A|^{-1/2}\geq |A+A|^{2}|AA||A|^{-1}, $$
and 
$$(|A+A||AA|)^{3/2}|A|^{-1/2}\geq |A+A||AA|^{2}|A|^{-1}.$$
Thus we finish the proof.
\end{proof}
\begin{cor}
\label{cor:ERRS1}
Let $A\subseteq \F_p$. If $|A|^2|AA|\ll p^{2}$ then
$$E^+(A)\ll |AA|^{3/2}|A|.$$
\end{cor}

We will require an iterative inequality for higher order energies to be used in the proof of  Theorem~\ref{thm:two}.
\begin{lemma}
\label{lem:EAA}
For integer $k\ge 2$ and a subset $A\subseteq \F_q$ we let $T_k(A)$ count the number of solutions to the equation
\begin{align}
a_1+\dots+a_k=a_{k+1}+\dots+a_{2k}, \quad a_1,\dots,a_{2k}\in A.
\end{align} 
Suppose $A$ satisfies
\begin{align}
\label{eq:JAcardcond1}
|A||(k-1)A||AA|\le p^{2},
\end{align}
then we have 
\begin{align*}
T_k(A)\lesssim |A|^{k-3/2}T_{k-1}(A)^{1/2}|AA|^{3/2}+|A|^{2k-3}|AA|+\frac{T_{k-1}(A)|AA|^2}{|A|}.
\end{align*}
\end{lemma}

\begin{proof}
For $\lambda \in (k-1)A$ we define
\begin{align*}
r(\lambda)=|\{ (a_1,\dots,a_{k-1})\in A\times \dots \times A \ : a_1+\dots+a_{k-1}=\lambda\}|.
\end{align*}
Then we have 
$$T_{k}(A)=\sum_{x} (A* r)(x)^{2}.$$
Now we take a dyadic decomposition for $r$. For  integer $j\ge 1$ let 
$$J(j)=\{ \lambda \in (k-1)A \ : 2^{j-1}\le r(\lambda)<2^j\}.$$ 
Then 
$$(A* r)(x) \ll \sum_{1\leq j\leq \log |A|} 2^j (A*J(j))(x).$$
By Cauchy-Schwarz inequality 
$$(A* r)(x)^2\lesssim \sum_{1\leq j\leq \log |A|} 2^{2j} (A*J(j))(x)^{2}.$$
Thus 
\begin{align*}
T_k(A)\lesssim \sum_{1\leq j\leq \log |A|} \sum_{x }2^{2j} (A*J(j))(x)^{2},
\end{align*}
and hence there exists some  $1\le i_0\ll \log{|A|}$ such that 
\begin{align}\label{eq:NewTk}
T_k(A)\lesssim 2^{2i_{0}} E^{+}(A, J(i_{0})).
\end{align}
By Lemma \ref{lem:ERRS}
\begin{align}
\label{eq:newEJ0}
E^{+}(A, J(i_0)) &\ll (|J(i_0)||AA|)^{3/2}|A|^{-1/2} \\ \nonumber
&+ \max\{|J(i_0)|,|AA|\}|J(i_0)||AA||A|^{-1},
\end{align}
provided 
\begin{align}
\label{eq:JA}
|J(i_0)||A||AA|\le p^{2}.
\end{align}
Since $J(i_0)\subseteq (k-1)A,$ the inequality~\eqref{eq:JA} is satisfied by~\eqref{eq:JAcardcond1}. By~\eqref{eq:NewTk} and~\eqref{eq:newEJ0}
\begin{align*}
T_k(A)&\lesssim \frac{(2^{i_0}|J(i_0)|)(2^{i_0}|J(i_0)|^{1/2})|AA|^{3/2}}{|A|^{1/2}}+\frac{(2^{2i_0}|J(i_0)|^2)|AA|}{|A|} 
\\ & +\frac{(2^{2i_0}|J(i_0)|)|AA|^2}{|A|},
\end{align*}
and since 
\begin{align*}
2^{i_0}|J(i_0)|\ll |A|^{(k-1)}, \quad 2^{2i_0}|J(i_0)|\ll T_{k-1}(A),
\end{align*}
we get 
\begin{align*}
T_k(A)\lesssim |A|^{k-3/2}T_{k-1}(A)^{1/2}|AA|^{3/2}+|A|^{2k-3}|AA|+\frac{T_{k-1}(A)|AA|^2}{|A|},
\end{align*}
which completes the proof.
\end{proof}

\section{Proof of Theorem \ref{thm:one}}
We consider the case $A+A$, a similar argument applies to $A-A$.  Assuming $A$ satisfies 
\begin{align}
\label{eq:Acardcond1}
|A|\le p^{64/117},
\end{align}
we consider two cases. Suppose first that 
\begin{align}
\label{eq:onecase1}
|A+A|^2|AA|\ll p^{2}.
\end{align}
By Lemma~\ref{lem:PlRu}, we can identify a subset $B \subset A$ satifying
\begin{align}
\label{eq:A'dense}
|B|\gg |A|
\end{align}
 and
\begin{equation}
\label{eqn:PR3A}
|B+B+B| \ll \frac{|A+A|^{2}}{|A|}.
\end{equation}
By~\eqref{eq:A'dense}, in order to prove Theorem~\ref{thm:one} it is sufficient to show that 
$$\max\{|B+B|,|BB|\}\gtrsim |B|^{39/32}.$$
Let 
\begin{align}
\label{eq:Pdef}
P=\left\{x\in B+B: r_{B+B}(x)\geq \frac{1}{2}\frac{|B|^{2}}{|B+B|}\right\},
\end{align}
so that
$$\sum_{x\in P}r_{B+B}(x)\gg|B|^{2}.$$
Applying Lemma \ref{lem:S} we have 
$$|B|^{8}\ll E^{+}_{4}(B)E^{+}(P),$$
and by Lemma \ref{lem:SS}
\begin{equation}\label{eq:first}
 |B|^{6}\lesssim |BB|^{2} E^{+}(P).
 \end{equation} 
It remains to consider $E^{+}(P).$ Recalling~\eqref{eq:Pdef}, we see that for any $x\in \F_p$, 
$$\frac{|B|^{2}}{|B+B|}P(x)\ll (B*B)(x),$$
and hence 
$$(P*P)(x)\ll \frac{|B+B|}{|B|^{2}}(B*B*P)(x).$$
Thus 
$$E^+(P)=\sum_{x}(P*P)(x)^{2}\lesssim \frac{|B+B|^{2}}{|B|^{4}}\sum_{x}(B*B*P)(x)^{2}.$$
Taking a dyadic decomposition for the function $(B*P)(x)$, there exists some real number $\Delta$ satisfying
$$1\leq \Delta \leq |B|,$$
 such that defining
$$T=\{x\in B+P: \Delta\leq (B*P)(x)<2\Delta\},$$
we have
\begin{equation}
\label{eq:PP}
\sum_{x}(P*P)(x)^{2}\lesssim \frac{|B+B|^{2}}{|B|^{4}} \Delta^{2} \sum_{x} (B*T)(x)^{2}=\frac{|B+B|^{2}}{|B|^{4}} \Delta^{2}E^+(B,T).
\end{equation}
Since $T\subseteq B+B+B$, by~\eqref{eq:onecase1} and \eqref{eqn:PR3A}  we have
\begin{align}
\label{eq:Acond1}
|B||B+B+B||BB|\ll p^{2},
\end{align}
and hence by Lemma \ref{lem:ERRS}  
\begin{align}
\label{eq:energy}
E^{+}(B, T)\ll & |T|^{3/2}  |BB|^{3/2}|B|^{-1/2} +|T|^{2}|BB||B|^{-1}\\ \nonumber
&+ |T||BB|^{2}|B|^{-1}.
\end{align}
This gives
\begin{align*}
&\sum_{x}(P*P)(x)^{2}\lesssim \frac{|B+B|^{2}}{|B|^{4}}(\Delta|T|)(\Delta |T|^{1/2})|BB|^{3/2}|B|^{-1/2} \\ & \quad \quad +
\frac{|B+B|^{2}}{|B|^{4}}(\Delta|T|)^2|BB||B|^{-1}+\frac{|B+B|^{2}}{|B|^{4}}(\Delta^2|T|)|BB|^{2}|B|^{-1}.
\end{align*}
Since 
$$\Delta |T|\ll |B||P|, \quad  \Delta^{2} |T|\ll E^{+}(B, P),$$
and $P\subseteq B+B$ the above simplifies to 
\begin{align}
\label{eq:E+Pcases}
&E^+(P)\lesssim \frac{|B+B|^{3}|BB|^{3/2}E^+(B,B+B)^{1/2}}{|B|^{7/2}} \\ & \quad \quad +
\frac{|B+B|^{4}|BB|}{|B|^{3}}+\frac{|B+B|^{2}|BB|^2E^+(B,B+B)}{|B|^{5}}. \nonumber
\end{align}
We next proceed on a case by case basis depending on which term in~\eqref{eq:E+Pcases} dominates. Suppose first that 
\begin{align*}
E^+(P)\lesssim \frac{|B+B|^{3}|BB|^{3/2}E^+(B,B+B)^{1/2}}{|B|^{7/2}}.
\end{align*} 
The assumption~\eqref{eq:onecase1} implies that the conditions  of Corollary \ref{cor:ERRS} are satisfied and hence
\begin{align}
E^{+}(P) 
&\lesssim\frac{|B+B|^{15/4}|BB|^{9/4}}{|B|^{15/4}}.
\end{align}
Combining with \eqref{eq:first}  we obtain
\begin{equation}
|B|^{39}\lesssim |B+B|^{15}|BB|^{17},
\end{equation}
which gives the required result. 

 Suppose next that  
\begin{align}
E^{+}(P) & \lesssim \frac{|B+B|^{4}|BB|}{|B|^{3}}.
\end{align}
Combining with $\eqref{eq:first}$ we obtain 
$$|B|^{9}\lesssim |B+B|^{4}|BB|^{3},$$
which gives better bound than $39/32$.

Finally, consider when
\begin{align}
E^{+}(P) &\lesssim \frac{|B+B|^{2}|BB|^2E^+(B,B+B)}{|B|^{5}}.
\end{align}
By Corollary~\ref{cor:ERRS}  we have 
\begin{align}
E^{+}(P) & \lesssim\frac{|B+B|^{7/2}|BB|^{7/2}}{|B|^{11/2}},
\end{align}
and hence by \eqref{eq:first}  
$$|B|^{23}\lesssim |B+B|^{7}|BB|^{11},$$
which gives better bound than $39/32$ and this finishes the proof in the case 
$$|A+A|^2|AA|\le p^{2}.$$
 Suppose next that 
\begin{align*}
|A+A|^2|AA|\ge p^{2}.
\end{align*}
By~\eqref{eq:Acardcond1} 
\begin{align*}
|A+A|^2|AA|\ge |A|^{117/32},
\end{align*}
and hence 
$$\max\{|A+A|,|AA|\}\ge |A|^{39/32},$$
which completes the proof.
\section{Proof of Theroem~\ref{thm:SP}}
Suppose $A$ satisfies
\begin{align}
\label{eq:twoAcard}
|A|\le p^{32/55},
\end{align} 
and consider two cases. Suppose first that 
\begin{align*}
|A-A||AA|A|\le p^{2}.
\end{align*}
By Lemma~\ref{lem:KR}, Lemma~\ref{lem:lower} and Corollary~\ref{cor:ERRS}  we get
$$
\frac{|A|^{8}}{|A-A|} \ll (|A|^{2}|AA|^{4/3})(|A-A|^{3/2}|AA|^{3/2}|A|^{-1/2}),
$$
which reduces to 
$$
|A-A|^{15}|AA|^{17}\gg |A|^{39},
$$
and gives the required result. If
\begin{align*}
|A-A||AA||A|\ge p^{2},
\end{align*}
then by~\eqref{eq:twoAcard}
\begin{align*}
|A-A||AA|\ge |A|^{39/16},
\end{align*}
and gives the required result.
\section{Proof of Theorem \ref{thm:two}}
Let $A$ satisfy
\begin{align}
\label{eq:thmkAcond}
|A|\le p^{(4-3\times 2^{-k})/(7-16\times2^{-k})},
\end{align}
and consider two cases. Suppose first that 
\begin{align}
\label{eq:thmkcase1}
|A||(k-1)A||AA|\le p^2.
\end{align}
We fix an integer $k\ge 3$ and consider two subcases. Suppose first that for all integers $3\le j \le k$ we have 
\begin{align*}
|A|^{j-3/2}T_{j-1}(A)^{1/2}|AA|^{3/2}\ge \max\left\{|A|^{2j-3}|AA|,\frac{T_{j-1}(A)|AA|^2}{|A|}\right\}. 
\end{align*}
By~\eqref{eq:thmkcase1} and Lemma~\ref{lem:EAA}, this implies that for each $3\le j \le k$ we have 
\begin{align*}
T_j(A)\lesssim |A|^{j}\left(\frac{|AA|}{|A|}\right)^{3/2}T_{j-1}(A)^{1/2},
\end{align*}
and hence by induction on $j$
\begin{align*}
T_k(A)\lesssim |A|^{k+(k-1)/2+\dots+(k-j+1)/2^{j-1}}\left(\frac{|AA|}{|A|}\right)^{3/2(1+1/2+\dots+1/2^{j-1})}T_{k-j}(A)^{1/2^j}.
\end{align*}
Taking $j=k-2$ and using Corollary~\ref{cor:ERRS1} gives
\begin{align*}
T_k(A)&\lesssim |A|^{k+(k-1)/2+\dots+3/2^{k-3}}\left(\frac{|AA|}{|A|}\right)^{3/2(1+1/2+\dots+1/2^{k-3})}E^{+}(A)^{1/2^{k-2}} \\
&\lesssim |A|^{k+(k-1)/2+\dots+3/2^{k-3}}\left(\frac{|AA|}{|A|}\right)^{3/2(1+1/2+\dots+1/2^{k-3})}|A|^{3/2^{k-1}}|A|^{1/2^{k-2}}.
\end{align*}
Since
\begin{align}
\label{eq:mmm}
k+\frac{(k-1)}{2}+\dots+\frac{3}{2^{k-3}}=2k-2-2^{3-k},
\end{align}
and
\begin{align*}
1+1/2+\dots+1/2^{k-3}=2-2^{3-k},
\end{align*}
we get 
\begin{align}\label{eq:Tk}
T_{k}(A)\lesssim |A|^{2k-5+2^{3-k}}|AA|^{3(1-2^{1-k})}.
\end{align}
To show~\eqref{eq:mmm}  we use the  identity 
$$\sum_{m=1}^{n}\frac{m}{2^{m}}=2-\frac{n}{2^{n}}+\frac{1}{2^{n-1}},$$ 
which follows from the equation 
$$2t=t+1-\frac{n}{2^{n}}+\sum_{m=1}^{n-1}\frac{1}{2^{m}},$$
where 
$$t=\sum_{m=1}^{n}\frac{m}{2^{m}}.$$
For $x\in \F_p$ let
$$r_{A, k}(x)=|\{(x_{1}, \cdots, x_{k})\in A^{k}: x_{1}+\cdots+x_{k}=x\}|.$$ Then 
$$|A|^{k}=\sum_{x}r_{A, k}(x).$$ 
By the Cauchy-Schwarz inequality 
\begin{align*}
|A|^{2k}\le |kA|T_k(A),
\end{align*}
since
$$\sum_{x}r_{A, k}(x)^{2}=T_{k}(A).$$
Applying  \eqref{eq:Tk} we obtain
\begin{align*}
|A|^{5-2^{3-k}}\lesssim |kA||AA|^{3-3\times 2^{1-k}},
\end{align*}
which implies 
\begin{align}
\label{eq:case1}
\max\{|kA|,|AA|\} \gtrsim |A|^{(5-2^{3-k})/(4-3\times 2^{1-k})}.
\end{align}
Suppose next that there exists some $3\le j \le k$ with 
\begin{align*}
|A|^{j-3/2}T_{j-1}(A)^{1/2}|AA|^{3/2}\le \max\left\{|A|^{2j-3}|AA|,\frac{T_{j-1}(A)|AA|^2}{|A|}\right\}. 
\end{align*}
If 
\begin{align*}
|A|^{2j-3}|AA|\ge \frac{T_{j-1}(A)|AA|^2}{|A|},
\end{align*}
then by Lemma~\ref{lem:EAA}
\begin{align*}
T_j(A)\lesssim |A|^{2j-3}|AA|.
\end{align*}
Using the Cauchy-Schwarz inequality as before we get 
\begin{align*}
|A|^{2j}\lesssim |A|^{2j-3}|jA||AA|,
\end{align*}
which implies 
\begin{align*}
\max\{|kA|,|AA|\} \gtrsim |A|^{3/2},
\end{align*}
and is better than~\eqref{eq:case1}. If

\begin{align*}
\frac{ T_{j-1}(A)|AA|^2}{|A|} \ge |A|^{2j-3}|AA|,
\end{align*}
then 
$$T_j(A)\lesssim \frac{T_{j-1}(A)|AA|^2}{|A|}\le |A|^{2j-7}|AA|^2E^+(A),$$
and hence by Corollary~\ref{cor:ERRS1}
$$T_j(A)\lesssim  |A|^{2j-6}|AA|^{7/2}.$$
This implies that 
\begin{align*}
|A|^{6}\lesssim |jA||AA|^{7/2}
\end{align*}
and hence 
\begin{align*}
\max\{|kA|,|AA|\} \gtrsim |A|^{4/3},
\end{align*}
which is better than~\eqref{eq:case1}.

 Suppose next that
\begin{align*}
|A||(k-1)A||AA|\ge p^2.
\end{align*}
By~\eqref{eq:thmkAcond} this implies
\begin{align*}
|(k-1)A||AA|\ge |A|^{2(5-2^{3-k})/(4-3\times 2^{1-k})},
\end{align*}
which completes the proof.


\end{document}